\documentclass[a4paper, 10pt]{amsart}

\usepackage[T1]{fontenc}
\usepackage[utf8]{inputenc}

\usepackage{amssymb}
\usepackage{amsmath}
\usepackage{amsfonts}
\usepackage{amsthm}
\usepackage{mathtools}

\usepackage{fixltx2e} 
\usepackage{isomath} 

\usepackage{graphicx, xcolor}
\usepackage{comment}

\usepackage{url}%
\usepackage{hyperref}%
\hypersetup{
  colorlinks=false,
  linkbordercolor=gray,
  citebordercolor=gray,
  urlbordercolor=gray}%

\usepackage{booktabs}
\usepackage{geometry}
\usepackage{pdflscape}
\usepackage{tikz,pgfplots}
\usepackage{color}
\usepackage{float}
\usepackage{enumitem}

\newcommand{\E}{\mathbb{E}}

\newcommand{\bin}{\textnormal{binomial}}

\DeclareMathOperator{\Cov}{Cov}

\DeclareMathOperator{\Span}{Span}
\DeclareMathOperator{\Var}{Var}

\theoremstyle{plain}%
\newtheorem{theorem}{Theorem}[section]
\newtheorem{lemma}[theorem]{Lemma}
\newtheorem{proposition}[theorem]{Proposition}

\theoremstyle{definition}
\newtheorem{definition}[theorem]{Definition}
\newtheorem{example}[theorem]{Example}

\theoremstyle{remark}
\newtheorem{remark}[theorem]{Remark}

\title{A noise sensitivity theorem for Schreier graphs}
\author{Malin Palö Forsström}
\date{\today}

\usepackage{todonotes}

\definecolor{code_gray}{rgb}{0.5,0.5,0.5}
\definecolor{dark_gray}{rgb}{0.2,0.2,0.2}
\definecolor{gray}{rgb}{0.5,0.5,0.5}
\definecolor{light_gray}{rgb}{0.95,0.95,0.95}
\definecolor{dark_blue}{RGB}{29,29,102}
\definecolor{mid_blue}{RGB}{51,51,178}
\definecolor{light_blue}{RGB}{153,153,217}

\makeatletter



\begin{document}

\maketitle

\begin{abstract}
During the last 15 years, several extensions of the concept of noise sensitivity, first coined by Benjamini, Kalai and Schramm in~\cite{schramm2000}, have been studied. One such extension was studied in~\cite{mpf2014}, where the definition of noise sensitivity was extended to include  noise consisiting of sequences of irreducible and reversible Markov chains. In this paper we focus on the case where the Markov chain is a random walk on a Schreier graph, and show that a version of the Benjamini-Kalai-Schramm noise sensitivity theorem, connecting influences to noise sensitivity, holds in this setting. We then apply this result to give an alternative proof of one of the main results from a recent paper on exclusion sensitivity by Broman, Garban and Steif~\cite{bgs2013}.
\end{abstract}

\section{Introduction}

In~\cite{schramm2000}, Benjamini, Kalai and Schramm coined the term noise sensitivity, looking at how likely the output of a Boolean function were to be different at the starting point and the ending point of a continuous time  random walks on a Hamming cube. In the same paper they gave a result, now often called the Benjamini-Kalai-Schramm noise sensitivity theorem,  connecting influences to noise sensitivity, which they then used to show that percolation crossings of boxes in \( \mathbb{Z}^2 \) are very sensitive to resampling some of the edges in its domain.
Since this paper was published, several  analogues of this theorem have been proved in  slightly different settings, such as noise being modelled as biased random walks on Hamming cubes~\cite{abgm2014}, Brownian motion~\cite{kms2012b}, exclusion processes~\cite{bgs2013} and random walks on Caley graphs~\cite{rb2014}. Also, a quantitative version of the original result has been given by Keller and Kindler~\cite{kk2013}. Our main result is a version  of the Benjamini-Kalai-Schramm noise sensitivity theorem for sequences of random walks on Schreier graphs. 

To define what we mean by a \emph{Schreier graph}, let \( \mathcal{G}  \) be a finite group acting transitively on a finite set \( S \). If \( g \in \mathcal{G} \) and \( x \in S \) we will write \( x_g\) to denote the action of \( g \) on \( x \). Further, let \( U \) be a symmetric generating set of \( \mathcal{G} \), meaning that \( U \) generates \( \mathcal{G} \) and \( U = U^{-1} \). Then \( S \), \( \mathcal{G} \) and \( U \) generate the Schreier graph \( \mathcal{G}(S,\mathcal{G},U ) \); having vertex set \( S \) and an edge \( (x,y ) \) whenever there is \( u \in U \) with \( y = x_u \).

By a random walk on a Schreier graph \( G (S, \mathcal{G}, U) \) we mean the following. Let \( (t_k)_{k \geq 0} \) be a sequence of times with \( t_0 = 0 \) and \( t_k - t_{k-1}\sim \exp(1) \) independently  for all \( k \geq 1 \). For each \( k \geq 1 \), pick  an element \( u_k \in U \) independently of everything else and uniformly at random.  Pick \( X_0 \in S \) uniformly at random and set \( X_t = X_0 \) for all \( t \in [t_0, t_1 ) \). For any \( k \geq 1 \) and \( t \in [t_k, t_{k+1} ) \), set \( X_t = (X_{t_{k-1}})_{u_k} \). If \( (X_t)_{t \geq 0 }\) is chosen according to this rule, we say that \( (X_t)_{t \geq 0 } \) is a random walk on the Schreier graph \( G(S,\mathcal{G},U) \), and whenever we talk about a random walk in the setting of Schreier graphs we will think of the random walk as having been generated in this way. As all Schreier graphs are regular,  \( X_0 \) will always be chosen  according to the stationary  measure \( \pi \) of this random walk.

Whenever we have a Schreier graph \( G (S, \mathcal{G}, U) \), we get a natural notion of the influence \( I_u(f) \) of an element  \( u \in U \) on a function \( f \), by defining 
\[
I_u(f) = P \Bigl( f(X_0) \not = f((X_0)_u)\Bigr).
\]
In the special case when the Schreier graph  \( G (S, \mathcal{G}, U) \) is an \( n \)-dimensional Hamming cube (\( S =\mathbb{Z}_2^n\), \( \mathcal{G} = (\mathbb{Z}_2^n , \oplus)\), where \( \oplus \) is coordinatewise addition module 2, and \( U = \{ (1,0,0,\ldots),\, (0,1,0,\ldots),\, (0,0,1,0,\ldots),\ldots,(0,0, \ldots, 0,1) \)),  this definition coincides with the definition of influences used in e.g.~\cite{schramm2000} and~\cite{kkl1988}.

We are now ready to present our main result.

\begin{theorem}
Let \( G =G(S, \mathcal{G},U)\) be a Schreier graph and let \( f \colon S \to \{ 0,1 \} \). Let \( X = ( X_t)_{t \geq 0 } \) be a random walk on \( G \), let   \( \lambda_1 \) be the spectral gap of \( X \) and let \( \rho \) be the log-Sobolev constant of \( X \). Then for any \( r \in (0,1) \), any \( \Lambda>0  \) and any \( T > 0 \),
\begin{equation}
\Cov (f(X_0), f(X_{T}) 
\leq 
\frac{e^{-\Lambda \log (r ) / \rho    }}{2\lambda_1   } \cdot   \left(\frac{ \sum_{u\in U}  I_u(f)^2 }{|U|}\right)^{1/(1+r) }  + \Var(f) \cdot e^{- \Lambda T}.
\label{equation: main result}
\end{equation}
\label{theorem: main result}
\end{theorem}


To make the main result more concrete, we now present two examples of families of Schreier graphs, and note what the spectral gaps and the log-Sobolev constants are in both cases. The first of these two examples has the \( n \)-dimensional Hamming cube as a special case.

\begin{example}\label{example: hypercubes}
Let \(  \Omega_{m,n} \) denote the Schreier graph \(  G( \{ \mathbb{Z}_m^n, (\mathbb{Z}_m^n, \oplus), \{ e_k \}_{k=1}^n) \),  where \( \oplus \) is coordinate\-wise addition modulo \( m \), and \(e_k \) is the unique element in \( {\{ 0,1, \ldots, m-1 \}^n }\) which is zero everywhere except at the \( k\)th coordinate where it is one.  Note that when \( m = 2 \), \( \Omega_m^n \) is an \( n \)-dimensional Hamming cube. 
Then \( |U_n| = n \), and it is well known that for a fixed \( m \in \mathbb{N} \), the random walk \( X^{(n)} \) on  \( \Omega_{m,n} \) has spectral gap \( \lambda_1^{(n)} = \frac{1 - \cos(2\pi/m)}{n}\sim 2\pi^2/m^2n \) and  {log-Sobolev} constant \( \rho^{(n)} \geq 4\pi^2/5m^2n \)~(see e.g.~\cite{cs2003} and~\cite{dsc1996}). 

Fix \( \varepsilon > 0 \) and \( k > 0 \).
If we let \( T = \varepsilon n \), \( \Lambda = k/m^2n \), it follows from~\eqref{equation: main result} that for any  function \( f_n \colon \mathbb{Z}_m^n \to \{ 0,1 \} \) and any \( r^{(n)} \in (0,1)\),
\begin{equation*}
\Cov (f_n(X_0^{(n)}), f_n(X_{\varepsilon n}^{(n)}) 
\leq 
\frac{m^2ne^{-5k \log (r^{(n)} ) /2\pi^2   }}{\pi^2  } \cdot   \left(\frac{ \sum_{u\in U_n}  I_u(f_n)^2 }{n}\right)^{1/(1+r^{(n)}) } \!\!\!\! + \Var(f_n) \cdot e^{- \varepsilon  k/m^2}.
\end{equation*}
As this holds for all \( k > 0 \), it follows that for any sequence \( (f_n)_{n\geq 1 } \), \( f_n \colon \mathbb{Z}_n^m \ to \{ 0,1 \} \), 
\[
\lim_{n \to \infty} \Cov (f_n(X_0^{(n)}), f_n(X_{\varepsilon n}^{(n)})  = 0 
\]
for any \( \varepsilon > 0 \) whenever there is \( \delta > 0 \) such that for all large enough \( n \),
\begin{equation}\label{equation: small sum of influences}
\sum_{u \in U_n} I_u(f_n)^2 \leq n^{-\delta}
\end{equation}
When \( m = 2 \), this is the weaker version of the Benjamini-Kalai-Schramm noise sensitivity theorem presented in e.g.~\cite{gs2014}.

A commonly used example of a sequence of functions \( (f_n)_{n \geq 1 } \) for which  
\[
\sum_{u \in U_n} I_u(f_n)^2 \leq n^{-\delta}
\]
is the so called \emph{Tribes function} which  first studied by Ben-Or and Linial in~\cite{bl1987}. To define this function, consider the following scenario. In an area lives a total of \( \ell_n \) tribes, each having exactly \( k_n \) members, which are in constant conflict with eachother. To decide whether or not to start a war in the area, each tribe, independently of all other tribes, lets its members vote. A member votes 1 if she wants her tribe to start a war, and 0 else. A tribe starts a war if all its members votes for starting a war. The tribes function is the function \( f_n \colon \{ 0,1 \}^{\ell_nk_n} \to \{ 0 , 1 \} \) whose entries represent the individual votes, and whose output is 1 if a war is started by any tribe and 0 else.

To simplify the notation for influences, we will write \( I_i(f_n) \) instead of \( I_{e_i}(f_n) \). Then for each \( i \in [n] \), one can check that
\[
I_i(\textnormal{Tribes}_{\, \ell_n,k_n}) = \left( 1 -  \left( \frac{1}{2} \right)^{k_n} \right)^{\ell_n-1}   \left( \frac{1}{2} \right)^{k_n-1} ,
\]
implying that
\[
 \sum_i I_i(\textnormal{Tribes}_{\, \ell_n,k_n})^2 = \ell_n k_n  \left( 1 - \left( \frac{1}{2} \right)^{k_n} \right)^{2(\ell_n-1)}  \left( \frac{1}{2} \right)^{2(k_n-1)} \hspace{-3em}.
\]
In particular, if we set \( k_n = \lfloor \log_2 n - \log_2 \log_2 n \rfloor \), \( \ell_n = n/k_n \) and assume that \( \ell_n \) is an integer, we obtain
\begin{equation}
 \sum_i I_i(\textnormal{Tribes}_{\, \ell_n,k_n})^2 \asymp  \left( \frac{\log_2 n}{\sqrt{n}} \right)^2\!\!\!.
\label{IIIequation: tribes on a cube}
\end{equation}
This shows that \( (\textnormal{Tribes}_{\, \ell_n,k_n})_n \)  is a sequence of Boolean functions such that the sum of its squared influences is small enough for the inequality in~\eqref{equation: small sum of influences} to hold.
\end{example}

We now consider a second example of a sequence of Scheier graphs, which is covered neither by the results in~\cite{schramm2000} nor by the results in~\cite{rb2014}

\begin{example}\label{example: slices}
Let \( (m_n)_{n\geq 1 }  \) be a sequence of integers with \( m_n \asymp n \)  and let \( S_n \) be the symmetric group on \( n \) elements. Further, let et  \(  J_{n,m_n} \) denote the Schreier graph
\[
 G( \{ w \in \{ 0,1 \}^n\colon \| w \| = m_n \}, S_n, \{ \textnormal{transpositions of } \{ 1, 2,  \ldots, n \} \})
\]
where \( S_n \) acts on a sequence \( w \in \{ 0,1 \}^n \) by permuting its elements.
The graph \( J_{n, m_n} \) is  the so called \emph{Johnson graph} with parameters \( n \) and \(m_n \),  and is an example of a graph that is a Schreier graph but not a Caley graph.
For any \( n \) and \( m_n \),  \( |U_n| = n \), and the random walk \( X^{(n)} \) on \(  J_{n,m_n} \) has \( |U_n| = \binom{n}{2} \), spectral gap \( \lambda_1^{(n)} = 2/(n-1) \)~\cite{ds1987} and  log-Sobolev constant \( \rho^{(n)} = \Theta \bigl(1/n\log \frac{n(n-1)}{2m_n(n-m_n)} \bigr) \)~\cite{ly1998}. Set \( T_n = n \) and \( \Lambda^{(n)} = k /n \) for some \( k > 0 \). Then using~\eqref{equation: main result}, for any \( r^{(n)} \in (0,1 ) \) and any \( f_n \colon \{ w \in \{ 0,1 \}^n \colon \| w \| = m_n \} \to \{ 0,1 \}\) and some absolute constant \( C \)  we have
\begin{align*}
\Cov (f_n(X_0), f_n(X_{n}) 
&\leq 
\frac{e^{-kC \log (r^{(n)} ) \cdot \log \frac{n(n-1)}{2m_n(n-m_n)}    }}{2/(n-1)   } \cdot   \left(\frac{ \sum_{u\in U_n}  I_u(f_n)^2 }{\binom{n}{2}}\right)^{1/(1+r^{(n)}) } \!\!\!\! + \Var(f_n) \cdot e^{- k}
\\&=
\frac{e^{-kC \log (r^{(n)} ) \cdot \log \frac{n(n-1)}{2m_n(n-m_n)}    } }{2}\cdot   \left(\frac{ 2\sum_{u\in U_n}  I_u(f_n)^2 }{n (n-1)^{-r^{(n)} }}\right)^{1/(1+r^{(n)}) } \!\!\!\! + \Var(f_n) \cdot e^{- k}.
\end{align*}
As \( m_n \asymp n \), if we pick \( r^{(n)} \) constant, the  term 
\[
e^{-kC \log (r^{(n)} ) \cdot \log \frac{n(n-1)}{2m_n(n-m_n)}    }
\]
will be bounded from above as \( n \to \infty\). As \( r^{(n)} \) can be chosen to be arbitrarily close to zero, is follows that for any sequence \( f_n \colon \{ w \in \{ 0,1 \}^n \colon \| w \| = m_n \} \to \{ 0,1 \} \),
\[
\lim_{n \to \infty} \Cov (f_n(X_0), f_n(X_{n})  = 0
\]
if
\begin{equation}\label{equation: influences for Johnson graphs}
 \sum_{u\in U_n}  I_u(f_n)^2 < n^{1-2\delta} \asymp |U_n|^{1/2 - \delta}
\end{equation}
 for some \( \delta > 0 \).

As in the previous example, one can show that one sequence of functions that satisfies~\eqref{equation: influences for Johnson graphs} is the Tribes function defined in the previous example. 

\end{example}

The main motivation for trying to find a version of the Benjamini-Kalai-Schramm noise sensitivity theorem in the setting of random walks on Schreier graphs was that quite recently, in~\cite{kkl2009}, O'Donnell and Wimmer generalized another theorem involving influences, the so called \textsc{kkl} theorem, to Schreier graphs. Although our proof is closer to the proof of the original noise sensitivity theorem as presented in eg.~\cite{gs2014}, we will use the terminology and notation from~\cite{kkl2009}.

In the next section, we will present the main tools and techniques that are used in the proof. The proof of Theorem~\ref{theorem: main result} is then presented in Section 3. Finally, in Section 4, we will show that Theorem~\ref{theorem: main result} provides an alternative proof of a weaker version of a noise sensitivity theorem  for exclusion processes that was given in~\cite{bgs2013}.

\section{Notation and tools}

Recall from the introduction that given a finite group \( \mathcal{G}  \) which acts transitively on a finite set \( S \) and a symmetric generating set \( U \) of \( G \), we say that the graph with vertex set \( S \) and an edge between two vertices \( x \) and \( y \) whenever there is \( u \in U \) with \( y = x_u \) is the Schreier graph \( \mathcal{G}(S,\mathcal{G},U ) \). Note that all Schreier graphs are  connected, regular and undirected. 

Throughout this paper, we will  be concerned with  reversible and irreducible continuous time Markov chains \( X \) which are random walks on Schreier graphs. For any such Markov chain \( X \), we will let \( S \) denote the state space, \( Q_n = (q_{ij} )_{i,j \in S }\) denote its generator matrix and \( \pi  \) be its stationary distribution. As all Schreier graphs are regular, the measure \( \pi  \) will be a uniform measure on the state space. We write \( X_t  \) to denote the position of \( X  \) at time \( t \in \mathbb{R}_+ \), and will always assume that \( X_0 \) has been choosen according to \( \pi \).

Next, for all \( t \geq 0 \), let \( \smash{H_t } \)  denote the continuous time Markov semigroup of the Markov chain given by 
\[
H_t= \exp(tQ) .
\]
In other words, \( H_t \) operates on a function \( f \) with domain \( S \) by
\[
H_t f(\cdot) = \E[f(X_t) \mid X_0  = \cdot].
\]
For real valued functions \( f \) and \( g  \) with domain \( S \), we will use the inner product 
\[
 \langle f , g \rangle = \langle f,g \rangle_{\pi}= \E [ f(X_0)g(X_0) ].
\]
As \( X  \) is assumed to be reversible and irreducible, we can find a set, \( \{ \psi_ j \}_{j = 0}^{|S|-1} \) of eigenvectors to \( { -Q} \), with corresponding eigenvalues 
\begin{equation}
0 = \lambda_0 < \lambda_1 \leq \lambda_2 \leq \ldots \leq \lambda_{|S|-1}
\label{IIIequation: positive eigenvalues}
\end{equation}
such that  \( \{ \psi_ j\} \) is an orthonormal basis with respect to \( \langle \cdot, \cdot \rangle \) for the space of real valued functions on \( S \) (see e.g.~\cite{chung1996}). We will always let \( \psi_0 \equiv 1 \) be the eigenvector corresponding to \( \lambda_0 = 0 \).
The smallest nonzero eigenvalue, \( \lambda_1 \), is called the spectral gap of the Markov chain \( X \), and its inverse, \( t_{rel} \coloneqq 1/\lambda_1 \) is called the relaxation time.  
The eigenvectors \( \{ \psi_ j\} \) of \( -Q \) will be eigenvectors to \( H_t \) aswell, with corresponding eigenvalues \( \{ e^{-\lambda_j t} \} \).
Since the set \( \{ \psi_ j\} \) is an orthonormal basis for the set of real valued functions with domain \( S \), for any \( f \colon S \to \mathbb{R} \) we can write
\[
f = \sum_{j=0}^{|S|-1} \langle f, \psi_j \rangle \, \psi_j.
\]
To simplify notation, we will write  \( \hat f(j) \) instead of \( \langle f, \psi_i \rangle \). 
For  \( j=0,1,\ldots, |S|-1\), the terms \( \hat f(j)  \) are called the Fourier coefficients of \( f \) with respect to the basis \( \{ \psi_j \} \), and are useful for expressing several probabilistic quantities which will be of interest of us. One of the simplest such quantities is the expected value of \( f(X_0) \) for a function \( f \colon S \to \mathbb{R} \). Recalling that \( \psi_0 \equiv 1 \), this can be expressed as
\[
\E [f(X_0)] = \E[f(X_0)\cdot 1] = \langle f, 1 \rangle = \langle f, \psi_0 \rangle = \hat f(0).
\]
Although a little bit more complicated, using the orthonormality of the eigenvectors \( \{ \psi_i \} \),  for any \( t \geq 0 \) we  get that
\begin{equation*}
\begin{split}
 \E[f(X_0)f(X_t)]
&= \E[f(X_0)H_tf(X_0)] 
= \langle f, H_tf \rangle
= \Bigl\langle\sum_{i=0}^{|S|-1} \hat f(i)\psi_i , H_t \sum_{j=0}^{|S|-1} \hat f(j)\psi_j \Bigr\rangle
\\&= \Bigl\langle\sum_{i=0}^{|S|-1} \hat f(i)\psi_i , \sum_{j=0}^{|S|-1} \hat f(j)H_t\psi_j \Bigr\rangle 
= \Bigl\langle\sum_{i=0}^{|S|-1} \hat f(i)\psi_i , \sum_{j=0}^{|S|-1} \hat f(j)e^{-\lambda_j t}\psi_j \Bigr\rangle 
\\&= \sum_{i=0}^{|S|-1}  \sum_{j=0}^{|S|-1} \hat f(i)\hat f(j) e^{-\lambda_j t}\langle\psi_i , \psi_j \rangle 
= \sum_{j=0}^{|S|-1}   e^{-\lambda_j t}\hat f(i)^2
\end{split}
\end{equation*}
and consequently, that
\begin{equation}
\begin{split}
\Cov (f(X_0), f(X_t)) 
&= \E [f(X_0)f(X_t)] - \E[f(X_0)] \E[f(X_t)]
\\&= \E [f(X_0)f(X_t)] - \E[f(X_0)]^2
\\&= \sum_{j=1}^{|S|-1}    e^{-\lambda_jt}\hat f(j)^2.
\label{equation: covariance}
\end{split}
\end{equation}

Now recall that for any Schreier graph \( G (S, \mathcal{G}, U) \), any function \( f\colon S \to \{ 0 ,1 \} \) and any  \( u \in U \), we defined the influence of \( u \) on \( f \) by
\[
I_u(f) \coloneqq P(f(X_0) \not = f((X_0)_u)) .
\]
To connect this with the generator \( Q \) and the random walk \( X\) on \( G (S, \mathcal{G}, U) \), for any \( w \in S \) we define \( L_u f(w) \coloneqq  f(w)-f(w_u) \). With this notation \( -Q \)  can be written as
\[
-Q = \frac{1}{|U|} \sum_{u \in U} L_u
\]
and
\[
I_u(f) =P(f(X_0) \not = f((X_0)_u)) =\E [(f(X_0) - f((X_0)_u))^2]  =  \langle L_u f,  L_u f \rangle  .
 \]
Finally, for any \( u \in U \) and any \( f\colon S \to \mathbb{R} \),  note that
\begin{align*}
2 \langle L_u f ,f \rangle 
&= \E[(f(X_0)-f((X_0)_u))f(X_0)] + \E[(f(X_0)-f((X_0)_u))f(X_0)]
\\&= \E[(f(X_0)-f((X_0)_u))f(X_0)] + \E[(f((X_0)_u)-f(((X_0)_u)_u))f((X_0)_u)]
\\&= \E[(f(X_0)-f((X_0)_u))f(X_0)] + \E[(f((X_0)_u)-f(X_0))f((X_0)_u)]
\\&= \E[(f(X_0)-f((X_0)_u))f(X_0)] - \E[(f(X_0) - f((X_0)_u))f((X_0)_u)]
\\&= \E[(f(X_0)-f((X_0)_u))(f(X_0)-f((X_0)_u))]
\\&= \langle L_u f , L_u f \rangle .
\end{align*}

In the proof of Theorem~\ref{theorem: main result}, we will use the following so called \emph{hypercontractive property} of the operator \( H_t \) (see e.g.~\cite{kkl2009},~\cite{dsc1996} and ~\cite{g1975}).
\begin{theorem}\label{theorem: sobolev}
Let \( \rho \) be the log-Sobolev constant for the random walk \( (X_t)_{t \geq 0 } \) on the Schreier graph \( G  (S,\mathcal{G},U) \). Then for all \( p \) and \( q \) satisfying \( 1 \leq p \leq q \leq \infty \) and \( \frac{q-1}{p-1} \leq \exp(2\rho t) \), and all \( f \in L^2(X) \), we have
\[
\| H_t f \|_q \leq \| f \|_p,
\]
where  \( \| \cdot \|_p = \E[|\cdot|^p]^{1/p} \), \( \| \cdot \|_q = \E[|\cdot|^q]^{1/q} \) and \( \| \cdot \|_2 = \E[|\cdot|^2]^{1/2} \).
\end{theorem}
The  \emph{log-Sobolev constant} of a Schreier graph  \( {G}(S, \mathcal{G}, U) \) is the largest constant such that for all nonconstant functions \( f \colon S \to \mathbb{R} \),
\[
\frac{1}{|U|} \sum_{u \in U} I_u(f) \geq \frac{\rho}{2} \cdot \Cov \left(   f(X_0^{(n)})^2, \log((f(X_0^{(n)})^2)   \right).
\]

We will  only use Theorem~\ref{theorem: sobolev} for \( q = 2 \), in which case the only  condition on \( p \) is that \(p \geq 1 + \exp(-2\rho t)\).

In the rest of this section, we will introduce the definition of noise sensitivity used in~\cite{schramm2000} and give an analogue of this definition in the setting of Schreier graphs. The main reason for introducing this concept is to give some terminology for the behaviour of the left hand side of~\eqref{equation: main result} for sequences \( (f_n)_{n \geq 1} \), and also to give some perspective on what kind of convergence of this term we are interested in.

\begin{definition}
Let  \( (X^{(n)})_{n\geq 1} \) be a sequence of reversible and irreducible Markov chains, with state spaces \( (S^{(n)})_{n \geq 1} \) and stationary distributions \( (\pi_n)_{n \geq 1} \), and let \( (  T_n)_{n \geq 1} \) be a sequence of real positive numbers. A sequence \( (f_n )_{n \geq 1} \)  of Boolean functions with \( f_n \colon S^{(n)} \to \{ 0,1 \} \) is said to be \emph{noise sensitive} with respect to \( (  X^{(n)},T_n)_{n \geq 1} \)   if for all \( \varepsilon > 0 \)
\begin{equation}
\lim_{n \to \infty} \Cov \left(f_n(X_0^{(n)}), f_n(X_{\varepsilon T_n}^{(n)}) \right)= 0.
\label{def: possible definition}
\end{equation}
\label{definition: noise sensitivity 2}
\end{definition}%

It is easy to show that if \( (X^{(n)})_{n \geq 1} \) is a sequence of reversible and irreducible continuous time Markov chains and  \( (f_n)_{n\geq 1} \) is a sequence of Boolean functions with domains \( (S^{(n)}) \)  such that \( \lim_{n\to \infty} \Var(f_n) = 0 \), then \( (f_n)_{n\geq 1} \) will automatically be noise sensitive with respect to \( (X^{(n)})_{n \geq 1} \). For this reason, a common restriction is to consider only sequences of Boolean functions which satisfy
\[
\lim_{n\to \infty} \Var(f_n) > 0.
\]
If this holds for  a sequence \( (f_n)_{n\geq 1} \) of Boolean functions, we say that \( (f_n)_{n\geq 1} \) is \emph{nondegenerate}. This property holds in both the previously given examples.

Using the eigenvalues of the generator \(-Q_n \), we will now give another characterization of noise sensitivity. Both this lemma and its proof are completely analogous to the first part of Theorem~1.9 in~\cite{schramm2000}.

\begin{lemma}
A sequence of Boolean  functions \( ( f_n)_{n \geq 1 } \), \( f_n \colon S^{(n)} \to \{ 0,1 \} \), is noise sensitive with respect to \( ( X^{(n)},T_n )_{n \geq 1}  \) if and only if for all \( \varepsilon > 0 \),
\begin{equation}
\lim_{n \to \infty} \sum_{i=1}^{|S^{(n)}|-1} e^{-\varepsilon T_n\lambda _i^{(n)}} \hat f_n(i)^2 = 0.
\end{equation}
Consequently, \( ( f_n)_{n \geq 1 } \) is noise sensitive with respect to \( ( X^{(n)},T_n )_{n \geq 1}  \) if and only if for all \( k > 0 \),
\begin{equation}
\lim_{n \to \infty} \sum_{i \colon \lambda_i^{(n)} < k /T_n} \hat f_n(i)^2=0.
\label{eq: epsilon dependencyIII}
\end{equation}
\label{lemma:noise_sensitivity_equiv}
\label{lemma: general finite sums}
\end{lemma}


\begin{proof}[Proof of Lemma~\ref{lemma:noise_sensitivity_equiv}]

Fix \( \varepsilon > 0 \) and let  \( t = \varepsilon T_n \). Then from~\ref{equation: covariance} it follows that 
\begin{equation}
\begin{split}
\Cov( f_n(X_0^{(n)}), f_n(X_{\varepsilon T_n}^{(n)}))  
&= 
\E \left[ f_n(X_0^{(n)})f_n(X_{\varepsilon T_n}^{(n)}) \right] -\E \left[ f_n (X_0^{(n)}) \right]^2  
\\&= 
\sum_{j=1}^{\mathclap{|S^{(n)}|-1}} e^{{-\varepsilon \lambda_j^{(n)}T_n} } \hat f_n(j)^2 .
\label{IIIequation: spectral formulation of noise sensitivity}
\end{split}
\end{equation}
For any \( \varepsilon > 0 \), it is easy to see that the left hand side of~\eqref{IIIequation: spectral formulation of noise sensitivity} tends to zero as \( n \to \infty \) if and only if~\eqref{eq: epsilon dependencyIII} holds. From this the desired conclusion follows.
\end{proof}

\begin{remark}
By the last lines of the proof above it follows that  if a sequence of functions satisfies~\eqref{def: possible definition} for one \( \varepsilon > 0 \), then it does so for all \( \varepsilon > 0 \), i.e. the proof of Lemma~\ref{lemma:noise_sensitivity_equiv} in fact shows that a sequence of Boolean functions \( (f_n)_{n \geq 1} \) is noise sensitive with respect to \( (X^{(n)},T_n)_{n \geq 1} \) if and only if
\begin{equation*}
\lim_{n \to \infty}\Cov( f_n(X_0^{(n)}), f_n(X_{ T_n}^{(n)}))  = 0.
\end{equation*}
\end{remark}

The concept of noise sensitivity relates to our main result as from this it follows that a sequence of Boolean functions \( f_n \colon S^{(n)} \to \{ 0,1 \} \) is noise sensitive  if 
\[
\lim_{n \to \infty}\sum_{u \in U_n} I_u(f_n)^2.
\]
In some cases, this makes it simpler  to show that a sequence of Boolean functions is noise sensitive, as it reduces a question about a process to a question about the geometry of a function, as the influences do not depend on time. In~\cite{schramm2000}, this was the idea that enabled Benjamini, Kalai and Schramm to show that the sequence of indicator functions of percolation crossings of a sequence of rectangles of increasing size is sensitive to noise.

\section{Proof of main result}

The main purpose of this section is to give a proof of Theorem~\ref{theorem: main result}. To this end, we will first state and prove a lemma that will be needed in  the proof. This result provides an analogue of a result stating that for the Hamming cube \( \Omega_{2,n} \), for any \( f \colon \{ 0,1 \}^n \to \{ 0,1 \} \) and any \( i \in \{ 1,2,\ldots, n \} \), 
\begin{equation}
2 \lambda_i^{(n)} |U_n| \hat f(i)^2  = \sum_{u \in U} \langle L_u  f , \psi_i^{(n)} \rangle^2
\label{IIIequation: stronger version of lemma on the Hamming cube}
\end{equation}
(see e.g. the proof of Proposition V.7 in~\cite{gs2014}).

Whenever \( \Psi = \Psi_\lambda \) is the span of all eigenvectors corresponding to some eigenvalue \( \lambda \), we say that \( \Psi \) is the eigenspace corresponding to \( \lambda \). Using this notation, we can now state the revised form of~\eqref{IIIequation: stronger version of lemma on the Hamming cube}.

\begin{lemma}
\label{lemma:the improved first inequality} 
Let \( G (S,\mathcal{G},U) \) be a Schreier graph and let \( f\colon S\to \{ 0,1 \} \) be a Boolean function with domain \( S \). Let \( Q \) be the generator of the corresponding random walk. Further, let \( \Psi_{\lambda} \) be an eigen\-space of \( -Q \) and let \( \psi_1, \ldots, \psi_m \) be  an orthonormal basis  of \( \Psi_\lambda \). Then
\begin{equation}
\label{eq: geometric result}
\sum_{i \in \{ 1, \ldots, m \}} 2 \lambda |U| \langle f, \psi_i \rangle^2 = \sum_{i \in \{ 1, \ldots, m \}}  \sum_{u \in U} \langle L_u f, \psi_i \rangle^2\!\!.
\end{equation}
\end{lemma}

\begin{remark}
\eqref{IIIequation: stronger version of lemma on the Hamming cube} is clearly stronger than Lemma~\ref{lemma:the improved first inequality}, as the latter contains an additional sum over an eigenspace. The reason why we cannot apply the same proof as for the Hamming cube is  that the proof in that case  relies heavely on knowing the eigenvectors of the random walk, whereas in a more general setting, these are not known, requiring an altogether different proof.  In fact, \eqref{IIIequation: stronger version of lemma on the Hamming cube} does not hold for all Schreier graphs. To see this, consider the Schreier graph \( G  (S_4, \, S_4,\,  \{ (12), (13), (14), (23), (24), (34)\}  \), i.e. the graph generated by the transpositions of the symmetric group of order four. Then it is straight forward to check that the function
\begin{equation*}
\psi (\sigma) = -1+2 \cdot \mathbf{1}_{\sigma(1)\in \{ 1,2 \} }
\end{equation*} 
is an eigenvector with eigenvalue \(\lambda =  2/3 \) to \( -Q\). Set \( f(\sigma) = \mathbf{1}_{\sigma(1) = 1} \). Then \( \langle f, \psi \rangle = \frac{1}{4} \), implying that the summand in the left hand side of~\eqref{IIIequation: stronger version of lemma on the Hamming cube} corresponding to this eigenvector is
\begin{equation*}
 2 \lambda |U| \langle f, \psi \rangle^2 = 2 \cdot \frac{2}{3} \cdot { 4 \choose 2 } \cdot \left( \frac{1}{4}\right)^2 = \frac{1}{2}.
\end{equation*}
For each term in the inner sum on the right hand side of the same equation we have \( \langle L_{(ij)} f, \psi \rangle = \frac{1}{4} \) if \( i = 1 \) (we assume \( i < j \)) and 0 else, implying that the  sum on the right hand side of~\eqref{IIIequation: stronger version of lemma on the Hamming cube} equals
\begin{equation*}
\sum_{(ij) \in U} \langle L_{ (ij)} f, \psi \rangle^2 = 3 \cdot \left( \frac{1}{4} \right)^2 = \frac{3}{16}
\end{equation*}
which clearly does not equal \( \frac{1}{2} \). 
\end{remark}

To be able to give a proof of Lemma~\ref{lemma:the improved first inequality}, we will need the following lemma.

\begin{lemma}
For any eigenspace \( \Psi_\lambda\) of \( -Q \) , any orthonormal basis \( \psi_1 \), \ldots, \(\psi_m \) of \( \Psi_{\lambda} \) and any \( u \in U \), the set \( \{ \psi_{i,u}\}_{i\in \{ 1, \ldots, m} \} \), \( \psi_{i,u} (w) \coloneqq \psi_i(w_u) \), is also an orthonormal basis for \( \Psi_\lambda \).
\label{lemma: basis after rotation}
\end{lemma}

\begin{proof}
As 
\begin{equation*}
\langle \psi_{i,u}, \psi_{j,u}\rangle = \langle \psi_{i}, \psi_{j}\rangle
\end{equation*}
it is immediately clear that \( \{ \psi_{i,u} \}_{i} \) is an orthonormal set, so it remains to show that \( \psi_{j,u} \in \Psi_\lambda = \Span \{ \psi_i \}_{i} \) for any   \(j \in \{ 1, \ldots, m\} \). To obtain this, it is enough to show that \( \psi_{j,u} \) is an eigenvector of \( -Q \) with eigenvalue \( \lambda \). However this is immidiate, as
\begin{equation*}
\begin{split}
-Q \psi_{j,u} 
&= 
\sum_{k=0}^{|S^{(n)}|-1} \langle -Q \psi_{j,u}, \psi_{k,u} \rangle \psi_{k,u}
=
\sum_{k=0}^{|S^{(n)}|-1}  \langle -Q \psi_{j}, \psi_{k} \rangle \psi_{k,u}
\\&=
\sum_{k=0}^{|S^{(n)}|-1}  \langle \lambda \psi_{j}, \psi_{k} \rangle \psi_{k,u}
=
\sum_{k=0}^{|S^{(n)}|-1}  \lambda \langle \psi_{j}, \psi_{k} \rangle \psi_{k,u}
=
\lambda \psi_{j,u}.
\end{split}
\end{equation*}
\end{proof}

\begin{proof}[Proof of Lemma~\ref{lemma:the improved first inequality}]
For any \( f \colon S \to \{ 0,1 \} \) and any \( u \in U \), define the function \( f_u \colon S \to \{ 0,1 \} \) by \({ f_u(w) \coloneqq f(w_u)} \). Then, as \( L_u f = f - f_u \), for any \( i \in \{ 0,1, \ldots, |S|-1\}\) we have
\begin{equation*}
\begin{split}
\sum_{u \in U} \langle L_u f, \psi_i \rangle^2 
&=
\sum_{u\in U} \left\{ \langle f, \psi_i \rangle^2 - 2\langle f, \psi_i \rangle \langle f_u, \psi_i \rangle + \langle f_u, \psi_i \rangle^2 \right\}
\\&=
\sum_{u\in U} \left\{ 2\langle f, \psi_i \rangle\langle L_u f, \psi_i \rangle + \langle f_u, \psi_i \rangle^2 - \langle f, \psi_i \rangle^2 \right\}.
\end{split}
\end{equation*}
Morover, using that
\begin{equation*}
\begin{split}
\sum_{u \in U} \Bigl\langle L_u f, \psi_i \Bigr\rangle &= \Bigl\langle \sum_{u \in U} L_u f, \psi_i \Bigr\rangle 
= \Bigl\langle |U| -Q f, \psi_i \Bigr\rangle 
= |U| \Bigl\langle -Q \sum_j \hat f(j)  \psi_j, \psi_i \Bigr\rangle
\\&= |U|\sum_j \hat f(j)   \Bigl\langle -Q  \psi_j, \psi_i \Bigr\rangle
= |U| \sum_j \hat f(j) \lambda  \Bigl\langle \psi_j, \psi_i \Bigr\rangle
=\lambda |U|  \hat f(i) 
\end{split}
\end{equation*}
it follows that 
\begin{equation*}
\begin{split}
\sum_{u \in U} \langle L_u f, \psi_i \rangle^2  &= 
2 \lambda |U| \hat f(i)^2 + \sum_{u\in U} \left\{ \langle f_u, \psi_i \rangle^2 - \langle f, \psi_i \rangle^2 \right\}
\\&=
2 \lambda |U| \hat f(i)^2 + \sum_{u\in U}  \langle f_u, \psi_i \rangle^2  - |U| \langle f, \psi_i \rangle^2\!\!.
\end{split}
\end{equation*}
Now recall that what we want to prove is that 
\[
\sum_{i \in \{ 1,2, \ldots, m \} } 2\lambda |U| \langle f, \psi_i \rangle^2 = \sum_{i \in \{ 1,2, \ldots, m \}} \sum_{u \in U } \langle L_u f, \psi_i \rangle^2.
\]
This follows if we can show that
\begin{equation}
\label{IIIequation:projection}
\sum_{i \in \{ 1, \ldots, m \}}  \sum_{u\in U}  \langle f_u, \psi_i \rangle^2  =  |U|  \sum_{i \in \{ 1, \ldots, m \}}  \langle f, \psi_i \rangle^2.
\end{equation}
As all sums in~\eqref{IIIequation:projection} are finite, for the left hand side of this equation we have
\begin{equation*}
\begin{split}
\sum_{i \in \{ 1, \ldots, m \}} \sum_{u\in U}  \langle f_u, \psi_i \rangle^2  
&=
\sum_{i \in \{ 1, \ldots, m \}}  \sum_{u\in U}  \langle f(w_u) \psi_i(w) \rangle^2  
\\&=
\sum_{i \in \{ 1, \ldots, m \}}  \sum_{u\in U} \langle f(w) \psi_i(w_{u^{-1}}) \rangle^2\!\! .
\end{split}
\end{equation*}
As \( U = U^{-1} \) by  definition, it follows that 
\begin{equation*}
\begin{split}
\sum_{i \in \{ 1, \ldots, m \}} \sum_{u\in U}  \langle f_u, \psi_i \rangle^2  
&=
\sum_{i \in \{ 1, \ldots, m \}}  \sum_{u\in U}  \langle f(w) \psi_i(w_{u}) \rangle^2  
\\&= 
\sum_{i \in \{ 1, \ldots, m \}}  \sum_{u\in U}  \langle f, \psi_{i,u} \rangle^2 
\\&=  
\sum_{u\in U}  \sum_{i \in \{ 1, \ldots, m \}}  \langle f, \psi_{i,u} \rangle^2 \!\!. 
\end{split}
\end{equation*}
By Lemma~\ref{lemma: basis after rotation}  \( \{ \psi_{i,u} \}_{i \in \{ 1, \ldots, m \}} \) is an orthonormal basis for \( \Psi \) for any  \( u\in U \). By combining this fact with Parseval's identity, we obtain
\begin{equation*}
\sum_{i \in \{ 1, \ldots, m \}}\langle f, \psi_{i,u} \rangle^2 = \sum_{i \in \{ 1, \ldots, m \}} \langle f, \psi_{i} \rangle^2  
\end{equation*}
from whic~\eqref{IIIequation:projection}, and thereby the  lemma, readily follows.
\end{proof}

We are now ready to give a proof of Theorem~\ref{theorem: main result}.

\begin{proof}[Proof of Theorem~\ref{theorem: main result}]
Note first that by Lemma~\ref{lemma:the improved first inequality};
\begin{equation*}
\sum_{i\geq 1 \colon  \lambda_i < \Lambda}  \hat f(i)^2 
\leq 
\frac{1}{2\lambda_1  |U|} \sum_{i=1}^{|S|-1} \sum_{u \in U} \langle L_u f, \psi_i  \rangle^2.
\end{equation*}
As by definition, for any \( t > 0 \) we have
\[
 \| H_t (L_u f) \|_2^2 = \sum_{i=1}^{|S|-1}e^{-2\lambda_j  t} \langle L_u f, \psi_j \rangle^2 
 \]
 we get
\begin{equation*}
\sum_{i\geq 1 \colon  \lambda_i < \Lambda}  \hat f(i)^2 
\leq 
\frac{e^{\,2\Lambda t }}{2\lambda_1  |U|} \sum_{u \in U} \|H_{t}  (L_u f)\|_2^2.
\end{equation*}
By the hypercontractivity principle, 
\begin{equation*}
\| H_{t} \left( L_u f \right) \|_2^2 \leq \| L_u f \| _p^2
\end{equation*}
whenever \(p \in [1,2]\) and \( p \geq 1+\exp(-2 \rho t) \), where \(\rho \) is the log-Sobolev constant for the random walk on \( G\). Now as \( f \in \{ 0,1 \} \), we have that \( \| L_u f \| _p^2 = \| L_u f \| _2^{4/p} \). Also, \( \| L_uf \|_2^2 = I_u(f) \). Putting these two observations together yields
\begin{equation*}
\| H_t \left( L_u f \right) \|_2^2 \leq \| L_u f \| _p^2 =  \| L_u f \| _2^{4/p} =  I_u(f)^{2/p}
\end{equation*}
which in turn implies that
\begin{equation*}
\begin{split}
\sum_{i\geq 1 \colon  \lambda_i < \Lambda}  \hat f(i)^2 
&\leq 
\frac{e^{\,2\Lambda  t }}{2\lambda_1   |U|}  \sum_{u\in U}  I_u(f)^{2/p}.
\end{split}
\end{equation*}
Multiplying by one and then applying Hölder's inequality, we obtain
\begin{equation*}
\begin{split}
\sum_{i\geq 1 \colon  \lambda_i < \Lambda}  \hat f(i)^2 
&\leq 
\frac{e^{\,2\Lambda  t }}{2\lambda_1 |U|}  \sum_{u\in U}  I_u(f)^{2/p} \cdot 1  
\leq
\frac{e^{\,2\Lambda  t}}{2\lambda_1  |U|}  \left( \sum_{u\in U}  I_u(f)^2 \right)^{1/p}  \left( \sum_{u\in U}  1^{\frac{p }{p-1}} \right)^{\frac{p -1}{p}}
\\ &=
\frac{e^{\,2\Lambda  t}}{2\lambda_1   |U|}  \left( \sum_{u\in U}  I_u(f)^2 \right)^{1/p}  |U|^{\frac{p -1}{p}}
=
\frac{e^{\,2\Lambda  t }}{2\lambda_1  } \cdot   \left(\frac{ \sum_{u\in U}  I_u(f)^2 }{|U|}\right)^{1/p}.
\end{split}
\end{equation*}
Since the only restriction on \( p \) is that  \( p \geq 1+\exp(-2 \rho t) \), we can set  \( p = 1+\exp(-2 \rho t) \). Using this, the previous inequality simplifies to
\begin{equation*}
\sum_{i\geq 1 \colon  \lambda_i< \Lambda}  \hat f(i)^2 
\leq 
\frac{e^{\,2\Lambda  t }}{2\lambda_1   } \cdot   \left(\frac{ \sum_{u\in U}  I_u(f)^2 }{|U|}\right)^{1/(1+\exp(-2 \rho t)) }.
\end{equation*}
Setting  \( r \coloneqq \exp( -2\rho t)  \) we obtain
\begin{equation*}
\sum_{i\geq 1 \colon  \lambda_i < \Lambda}  \hat f(i)^2 
\leq 
\frac{e^{-\Lambda  \log (r  /\rho   ) }}{2\lambda_1   } \cdot   \left(\frac{ \sum_{u\in U}  I_u(f)^2 }{|U|}\right)^{1/(1+r) }.
\end{equation*}

If we now use~\eqref{equation: covariance}, the desired inequality readily follows

\end{proof}

\section{Applying the noise sensitivity theorem to exclusion sensitivity}

The purpose of this section is to use the main result in the previous section, Theorem~\ref{theorem: main result},  to give a proof of a weaker version of Theorem 1.14 in~\cite{bgs2013}, which connects influences on \( \Omega_{2,n} \) to a concept which the authors call \emph{exclusion sensitivity}. To this end, we first define what we mean by exclusion sensitivity.
Let \( (f_n)_{n \geq 1 } \) be a sequence of Boolean functions with domain \( \mathbb{Z}_2^n \). Pick \( {X_0^{(n)}    }\) according to \(\pi_n\), where \( \pi_n \) is the uniform measure probability measure on \( \mathbb{Z}_2^n \), and let \( X^{(n)} \) be a random walk on \( J_{n,\|X_0^{(n)}\|} \), where \( J_{n,\|X_0^{(n)}\|} \), is the graph defined in Example~\ref{example: slices} . Equivalently, \( X^{(n)} \) is a random walk on \( \bigcup_{m = 0}^n J_{n,m}\) with \( X_0^{(n)} \sim \pi_n\).  We  say that \( (f_n )_{n \geq 1 } \) is \emph{exclusion sensitive} if for all \( \alpha > 0 \),
\[
\lim_{n \to \infty} \Cov(f_n(X_0^{(n)}), f_n(X^{(n)}_{\varepsilon n })) = 0.
\]

In~\cite{bgs2013}, Broman, Garban and Steif proved the following result.

\paragraph{\textbf{Theorem 1.14 in~\cite{bgs2013}}}%
\textit{Let \( (f_n)_{n \geq 1 } \), \( f_n \colon \mathbb{Z}_2^n \to \{ 0,1 \} \), be a sequence of Boolean functions such that 
\[
 \lim_{n \to 0} \sum_{i = 1}^n I_i(f_n)^2 = 0.
 \]
Then \( (f_n)_{n \geq 1 } \) is exclusion sensitive.}

We will prove the following weaker version of this result.

\begin{proposition}\label{proposition: exclusion sensitivity}
Let \( (f_n)_{n \geq 1 } \), \( f_n \colon \mathbb{Z}_2^n \to \{ 0,1 \} \), be a sequence of Boolean functions such that for some \( \delta > 0 \) and all large enough \( n \),
\[
 \sum_{i=1}^n I_i(f_n)^2  \leq n^{-\delta} \!\!\!\!\!.
 \]
Then \( (f_n)_{n \geq 1 } \) is exclusion sensitive.
\end{proposition}

The main motivation for trying to use Theorem~\ref{theorem: main result} to give a proof of something like  Proposition~\ref{proposition: exclusion sensitivity} is an exclusion process is very similar to a random walk on a Johnson graph \( J_{n, m_n} \) if we pick \( m_n \) at random. Also, one might guess that there should be some relation between the sum \( \sum_{i = 1}^n I_i(f)^2 \) and sums \( \sum_{(ij)} I_{(ij)}(f)^2 \). Roughly, the proof of Proposition~\ref{proposition: exclusion sensitivity}  that will be provided in this section simply formalize these ideas.

As we in the proof of this result will deal with several graphs simultaneously, we need to refine our notation from the previous sections. To this end, let \( \{ \psi_i^{(n)} \}_i\) be an orthonormal set of eigenvectors to the generator \(Q^{(n)} \) of  the random walk on \( \Omega_{2,n} \) with corresponding eigenvectors \( \{ \lambda_i^{(n)} \}_i \). Further, let \( \{ \psi_i^{(n,m)} \}_i \)  and \( \{ \lambda_i^{(n,m)} \}_i \)  be the corresponding sets for \( J_{n,m} \). Let \( \pi_n \) be the stationary distribution for the random walk on \( \Omega_{2,n} \) and let \( \pi_{n,m} \) be the stationary distribution for the random walk on \( J_{n,m} \).  Whenever we write \( X^{(n)} \), we will mean the random walk on \( J_{n,\|X_0^{(n)}\|} \), and whenever we have an expectation or probability containing the term \( X_0^{(n)} \), \( X_0^{(n)} \sim \pi_n \). Whenever we write \( I_i(f_n) \) for some \( i \in \{ 1,2, \ldots, n \} \) we mean \( P(f(X_0^{(n)}) \not = f((X_0^{(n)})_{e_i})) \), where \( X_0^{(n)} \sim \pi_n \). Similarly, for \( (ij) \in S_n \) we will write \( I_{(ij)}(f_n) \) to denote \( P(f(X_0^{(n)} ) \not = f((X_0^{(n)} )_{(ij)})) \) where \( X_0^{(n)} \sim \pi_n \). However, in addition to these notations we will write \( I_{(ij)}^{(m)}(f_n) \) to denote \( P(f(X_0^{(n)} ) \not = f((X_0^{(n)} )_{(ij)}) \mid \| X_0^{(n)}  \| = m) \).

Our first lemma relates the property of being exclusion sensitive with the property of being noise sensitive on slices \( J_{n,m} \).
\begin{lemma}\label{lemma: splitting}
For any \( n \geq 1 \), \( f\colon \mathbb{Z}_2^n \to \{ 0,1 \} \), and \( t \geq 0 \), 
\begin{equation}
\begin{split}
&\Cov (f_n(X_0^{(n)}), f_n(X^{(n)}_{t_{rel}^{(n)}}))
\\&\hspace{2em}=
\sum_{m=0}^n P(\| X_0^{(n)} \| = m) \cdot \Biggl\{ \E \left[ f(X_0^{(n)})f(X^{(n)}_{t_{rel}^n}) \mid \| X^{(n)}_0 \| = m \right] 
\\& \hspace{18em}- \E \left[f(X_0^{(n)}) \mid \| X_0^{(n)} \| = m\right]^2 \Biggr\}
\\&\hspace{4em}+
\Var \left(  \E \left[ f(X_0^{(n)}) \mid   \| X_0^{(n)} \| \right] \right).
\end{split}
\label{IIIequation: splitting variance}
\end{equation}
\end{lemma}

\begin{proof}
The desired result follows directly from the well known result stating that for three random variables \( X \), \( Y \) and \( Z \),
\[
\Cov (X,Y) = \E \left[\Cov(X,Y \mid Z) \right]+ \Cov \left( \E [X \mid Z], \E [Y \mid Z]\right)
\]
by setting \( X = f(X_0^{(n)}) \), \( Y = f(X_t^{(n)})  \) and \( Z = \| X_0^{(n)} \|  \).
\end{proof}

Note that the previous lemma shows that being exclusion sensitive exactly corresponds to being noise sensitive on almost all slices of the Hamming cube  and asymptotically having the same mean on almost all such slices.

We will  deal with the first and second term on the right hand side of~\eqref{IIIequation: splitting variance} separately. We begin with the first term, which is where we will use Theorem~\ref{theorem: main result}. To this end, we will first give a proof of the following lemma, which relates the sum \( \sum_{i=1}^n I_i(f_n)^2 \) with the sums \( \sum_{(ij) \in S_n} I_{(ij)}^{(m)}(f_n)^2 \) for \( m \in \{0, 1, \ldots, n \} \).

\begin{lemma}\label{lemma: good indices}
For any  \( {\alpha \in (0,1/2)} \),  
\begin{equation*}
P \left( \sum_{(ij) \in S_n}  I_{(ij)}^{(\|X_0^{(n)}\|)}(f_n)^2<  n \cdot \left( \sum_{i=1}^n I_i(f_n)\right)^{1-2\alpha} \right) \geq 1 - 4 \left( \sum_{i=1}^n I_i(f_n)\right)^\alpha  \!\!\!\!.
\end{equation*}
Consequently, if \( \sum_{i=1}^n I_i(f_n)^2 < n^{-\delta} \), then for any \( \alpha \in (0,1/2)\),
\begin{equation*}
P \left( \sum_{(ij) \in S_n}  I_{(ij)}^{(\|X_0^{(n)}\|)}(f_n)^2<  n^{1 - \delta(1-2\alpha)} \right) \geq 1 - 4 n^{-\alpha \delta}  \!\!\!\!\!\!\!\!.
\end{equation*}
\end{lemma}

\begin{proof}
Note first that
\begin{equation*}
\begin{split}
\sum_{(ij)\in S_n} I_{(ij)}(f_n)^2
&\leq
\sum_{(ij)\in S_n} \left( I_{i}(f_n) + I_j(f_n)\right) ^2
\\&\leq
2\sum_{(ij)\in S_n}   I_{i}(f_n)^2 + I_j(f_n)^2
\\&\leq
4n \sum_{i=1}^n   I_{i}(f_n)^2.
\end{split}
\end{equation*}
Moreover, if we let \( p_m^{(n)} \coloneqq P(\|X_0^{(n)}\| = m) \), then
\begin{equation*}
\begin{split}
\sum_{(ij)\in S_n} I_{(ij)}(f_n)^2
&=
\sum_{(ij)\in S_n} \left( \sum_{k = 0}^n p_k^{(n)} I_{(ij)}^{(k)}(f_n)\right) ^2
\\&=
\sum_{k = 0}^n p_k^{(n)} \sum_{(ij)\in S_n}  I_{(ij)}^{(k)}(f_n)  \left( \sum_{\ell = 0}^n p_\ell^{(n)} I_{(ij)}^{(\ell)}(f_n)\right).
\end{split}
\end{equation*}
Consequently
\begin{equation*}
\sum_{k = 0}^n p_k^{(n)} \sum_{(ij)\in S_n}  I_{(ij)}^{(k)}(f_n)  \left( \sum_{\ell = 0}^n p_\ell^{(n)} I_{(ij)}^{(\ell)}(f_n)\right)
\leq
4n \sum_{i=1}^n   I_{i}(f_n)^2 
\end{equation*}
Let \( A_n(a) \) be the set of all  \( k \in \{0,1, \ldots, n\}  \) such that 
\begin{equation*}
\sum_{(ij)\in S_n}  I_{(ij)}^{(k)}(f_n)  \left( \sum_{\ell=0}^n p_\ell^{(n)} I_{(ij)}^{(\ell)}(f_n)\right) < 4an \sum_{i=1}^n   I_{i}(f_n)^2\!.  
\end{equation*}
By picking \( a \) large, we can make the probability \( P( \| X_0^{(n)} \| \in A_n(a)) \) large. In particular, by Markov's inequality, 
\[
P( \| X_0^{(n)} \| \in A_n(a))  \geq 1 - \frac{1}{a} .
\]
Similarly, let \( B_n(b) \) be the set of all  \( k\in \{ 0,1, \ldots, n \} \) such that
\begin{equation*}
\sum_{(ij)\in S_n}  I_{(ij)}^{(k)}(f_n)^2
<
b \cdot \sum_{(ij)\in S_n}  I_{(ij)}^{(k)}(f_n)  \left( \sum_{\ell = 0}^np_\ell^{(n)} I_{(ij)}^{(\ell)}(f_n)\right) .
\end{equation*}
Then again by Markov's inequality, 
\[
P( \| X_0^{(n)} \| \in B_n(b))  \geq 1 - \frac{1}{b}.
\]
Summing up, we have showed that for any \( a,b > 0 \), if we pick an level \( m \) at random according to the law of \( \| X_0^{(n)} \| \), the probability is at least \( 1 - \frac{1}{a} - \frac{1}{b} \) that
\begin{equation*}
\sum_{(ij)\in S_n}  I_{(ij)}^{(k)}(f_n)^2< 4abn   \sum_{i=1}^n   I_{i}(f_n)^2  
\end{equation*}
Consequently, if we set \(  a = b =  \left( \sum_{i=1}^n I_i(f_n)^2  \right)^{-\alpha }/2\) for some \( \alpha \in (0,0.5) \), then the probability is at least \( 1 -    4 \left( \sum_{i=1}^n I_i(f_n)\right)^\alpha \) that
\begin{equation*}
\sum_{(ij)\in S_n}  I_{(ij)}^{(k)}(f_n)^2<  n \cdot \left( \sum_{i=1}^n   I_{i}(f_n)^2  \right)^{1 - 2\alpha}\!\!\!\!\!\!\!\!\!\!\!\!.
\end{equation*}

\end{proof}
Note that the previous lemma shows that if there is \( \delta \in (0,1) \) such that \( \sum_i I_i(f_n)^2 < n^{-\delta} \), then there is \( \delta' \in (0,1) \) such that asymptotically, 
\begin{equation*}
\sum_{(ij)\in S_n}  I_{(ij)}^{(k)}(f_n)^2<  n^{1-\delta'}
\end{equation*}
for almost all \( k \in \{ 0,1, \ldots, n \} \). In particular, we get the same \( \delta' \) for all such \( k \). This sets us up in a good position to use  Theorem~\ref{theorem: main result}.

\begin{lemma}\label{lemma: common bound}
Let  \( f_n \colon \mathbb{Z}_2^n \to \{ 0 ,1 \} \)  and let \( m \in \{ 0,1, \ldots n \} \) satisfy 
\begin{equation}
 2m(n-m) \geq \varepsilon n(n-1) .
\label{equation: mn condition}
\end{equation}
Further let
\[
 \hat f_{n,m}(i) \coloneqq \langle f_n, \psi_i^{(n,m)} \rangle = \E[f_n(X_0^{(n)}) \psi_i^{(n,m)}(X_0^{(n)}) \mid \| X_0^{(n)} \| = m ] .
\]
 If there is \( \delta > 0  \) such that
\[
\sum_{(ij)\in S_n} I_{(ij)}^{(m)} (f_n)^2 \leq |U_n|^{1/2-\delta}
\]
then
\begin{equation}
\sum_{i \geq 1\colon  \lambda_i^{(n)} < C\lambda_1^{(n)}} \hat f_{n,m}(i)^2 
\leq 
\frac{1}{2}\cdot \exp \left( C \cdot \log \varepsilon \cdot   \log \delta +  \log \frac{2n}{n-1}  \right)  \cdot  n^{
   1-  \frac{ 1+2\delta}{1 + \delta} 
}.
\end{equation}
Consequently, for any \( m \in \mathbb{N} \) such that \( 2m(n-m) \geq \varepsilon n(n-1) \) and any \( C > 0 \), there is a constant \( C' = C'(C, \delta, \varepsilon) \) such that 
\begin{equation*}
\begin{split}
\Cov_{ X_0^{(n)} \sim \pi_{n,m}}( f_n(X_0^{(n)}),f_n(X_{\alpha /\lambda_1^{(n,m)}}^{(n)}) )
&=
\sum_{i=1}^{\mathclap{|S^{(n)}|-1}} e^{{-\alpha \lambda_j^{(n,m)}/\lambda_1^{(n,m)}} } \hat f_{n,m}(i)^2 
\\ &\leq
\sum_{i\geq 1 \colon  \lambda_i^{(n,m)}<C\lambda_1^{(n,m)}}  \!\!\!\!\!\!\!\! \hat f_{n,m}(i)^2 
+
 e^{-\alpha C } 
\\ &\leq C' \cdot n^{1-\frac{1+2\delta}{1+\delta}} + e^{-\alpha C}.
\end{split}
\end{equation*}
\end{lemma}

\begin{remark}
Note that as \( \|X_0^{(n)} \| \sim \bin(n,0.5) \), 
\[ P(\|X_0^{(n)}\|(n-\|X_0^{(n)}\|) \geq \varepsilon  n^2 /2) \to 1. \]
This implies that if we consider a random walk on \( \bigcup_{m = 0}^n J_{n,m} \) and choose \( X_0^{(n)} \sim \pi_n \), then the lemma above gives us a convergence which is more or less uniform in \( m \).
\end{remark}

\begin{proof}
Note first that for any sequence \( ( \Lambda^{(n)})_{n \geq 1 } \) for positive real numbers, Theorem~\ref{theorem: main result} ensures that
\begin{equation*}
\sum_{i\geq 1 \colon  \lambda_i^{(n,m)} < \Lambda^{(n)}}  \hat f_{n,m}(i)^2 
\leq 
\inf_{r^{(n)}\in (0,1)} \frac{e^{-\Lambda^{(n)}/\rho^{(n,m)} \cdot   \log (r^{(n)} ) }}{2\lambda_1^{(n,m)}   } \cdot   \left(\frac{\sum_{(ij)} I_{(ij)}^{(m)}(f_n)^2 }{\binom{n}{2}}\right)^{1/(1+r^{(n)}) }
\end{equation*}
where \( \rho^{(n,m}) \) is the log-Sobolev constant for the random wak on \( J_{n,m} \). Recall that for \( J_{n,m} \), \( \lambda_1^{(n,m)} = 1/n \) and  \( \rho^{(n,m)} = 1/n\log \frac{n(n-1)}{2m(n-m)} \). Also, by assumption,
\[
 \sum_{(ij)} I_{(ij)}^{(m)} (f_n)^2 \leq {\binom{n}{2}}^{1/2-\delta}\!\!\!\!\!\!\!\!\!\!\!\!.
\]
 This implies that for any \( C > 0 \),
\begin{equation*}
\sum_{i\geq 1 \colon  \lambda_i^{(n)} < C \lambda_1^{(n)}}  \hat f_{n,m}(i)^2 
\leq 
\inf_{r^{(n)} \in (0,1)} \frac{e^{-C \cdot \log \frac{n(n-1)}{2m(n-m)} \cdot   \log (r^{(n)} ) }}{2/n   } \cdot   \left(\frac{ {\binom{n}{2}}^{1/2-\delta} }{{\binom{n}{2}}}\right)^{1/(1+r^{(n)}) }\hspace{-3.8em}.
\end{equation*}
Using the bound given by~\eqref{equation: mn condition} and simplifying, get the following upper bound on the right hand side of the previous equation.
\begin{equation*}
\frac{1}{2} \cdot \inf_{r^{(n)}\in (0,1)} \exp \left( 
C  \log \varepsilon \cdot   \log r^{(n)} + \frac{(1/2+\delta)\log  \frac{2n}{n-1}}{1 + r^{(n)}} + \left( 1 - \frac{1+2\delta}{1 + r^{(n)}} \right) \log n
\right)
\end{equation*}
Setting \( r^{(n)} = \delta \) the first claim of the lemma follows. The second claim of the lemma follows by imitating proof of Lemma~\ref{lemma: general finite sums}.
\end{proof}

The only thing which remains to do before we are set up to prove the Proposition~\ref{proposition: exclusion sensitivity} is to show that \( \lim_{n \to \infty} \Var ( \E [f_n(X_0^{(n)}) \mid   \| X_0^{(n)}\|  ]) = 0 \) if \( \lim_{n \to \infty} \sum_{i=1}^n \hat f_n(i)^2  = 0 \). This is the purpose of the following lemma.

\begin{lemma}\label{lemma: variance}
Let \( (f_n)_{n \geq 1 } \) be a sequence of Boolean functions for which \({ \lim_{n \to \infty} \sum_{i=1}^n I_i(f_n)^2 = 0} \). Then \( \lim_{n \to \infty } \Var(\E[f_n(X_0^{(n)}) \mid   \| X_0^{(n)} \|]) = 0 \).
\end{lemma}

\begin{proof}
Assume for contradiction that there exists some \( \varepsilon > 0 \) s.t. 
\begin{equation}
\lim_{n \to \infty}\Var \left( \E \left[f_n(X_0^{(n)} ) \mid  \| X_0^{(n)} \| \right] \right) >2\varepsilon^2. \label{IIIequation: contr. assumption}
\end{equation}
In practice, this might require taking subsequence, but to simplify notations, we will assume that this holds for all \( n \).

Set 
\[
 \mathcal{E}_m \coloneqq  \E [f(X_0^{(n)} ) \mid \| X_0^{(n)}  \| = m ]
\]
and define
\[
B_n \coloneqq  \left\{ w \in \Omega_n \colon  \left| \mathcal{E}_{\| w \|} - \mathcal{E}_{\lfloor n/2 \rfloor}  \right| > \varepsilon \right\}.
\]
Then
\begin{equation*}
\begin{split}
2\varepsilon^2 &< \Var \left( \E \left[f_n(X_0^{(n)}) \mid  \| X_0^{(n)} \| \right] \right) 
\\&= \E \left[ \left( \E \left[f_n(X_0^{(n)}) \mid  \| X_0^{(n)} \| \right] - \E[f(X_0^{(n)})] \right)^2 \right].
\end{split}
\end{equation*}
As for any real-valued random variable \( Z \), \( \E [Z] \) is the global minimum of the function \( z \mapsto \E [(Z-z)^2] \), we can bound the last expression in the previous equation by
\begin{equation*}
\E \left[ \left( \E \left[f_n(X_0^{(n)}) \mid  \| X_0^{(n)} \|  \right] - \E[f(X_0^{(n)}) \mid \| X_0^{(n)} \| =  \lfloor n/2 \rfloor] \right)^2 \right]
\end{equation*}
which is precisely equal to
\begin{equation}\label{equation: this quantity}
\E \left[ \left(\mathcal{E}_{\| X_0^{(n)} \| }- \mathcal{E}_{\lfloor n/2 \rfloor} \right)^2 \right].
\end{equation}
Now on the event \( B_n \), as \( f_n \) is Boolean, we can bound~\eqref{equation: this quantity} from above by 1. In the opposite situation, given the event \( B_n^c \), we can bound~\eqref{equation: this quantity}  from above by \( \varepsilon^2 \). This gives us the upper bound
\begin{equation*}
2\varepsilon^2 <  P(B_n) \cdot 1 + P(B_n^c)\cdot \varepsilon^2\!,
\end{equation*}
implying that
\[
P(B_n) \geq \varepsilon^2\!.
\]
From this it follows that there is a \( \alpha > 0 \) and a sequence \( ( \alpha_n )_{n \geq 1 } \), where \( \alpha_n \leq \alpha \), such that 
\[
\left| \mathcal{E}_{\lfloor n/2 \rfloor + \alpha_n \sqrt{n}}- \mathcal{E}_{\lfloor n/2 \rfloor}   \right| > \varepsilon
\]
for all \( n \geq 1 \).

Now note that
  \begin{equation*}
\begin{split}
\hspace{3em}&\hspace{-3em}\sum_{(w,w') \colon w \sim w'} (f_n(w) - f_n(w'))^2
\\&\geq 
\!\!\!\! \sum_{\ell = \lfloor n/2 \rfloor }^{\lfloor n/2 \rfloor + \alpha_n\sqrt{n}-1}  \sum_{(w,w') \colon w \sim w',\, \atop \| w \| = \ell, \, \| w' \| = \ell+1 } (f_n(w) - f_n(w'))^2
\\&\geq
\!\!\!\!  \sum_{\ell = \lfloor n/2 \rfloor }^{\lfloor n/2 \rfloor + \alpha_n\sqrt{n}-1} \left|\mathcal{E}_{\ell} \cdot \binom{n}{\ell} \cdot (n-l) - \mathcal{E}_{l+1} \cdot \binom{n}{\ell + 1} \cdot (\ell +1)\right|
\\&=
\!\!\!\!  \sum_{\ell = \lfloor n/2 \rfloor }^{\lfloor n/2 \rfloor + \alpha_n\sqrt{n}-1}  \binom{n}{\ell} \cdot (n-l)  \cdot  \left|\mathcal{E}_{\ell}   - \mathcal{E}_{l+1} \right|.
\end{split}
\end{equation*}
As for all \( \ell \in \{ \lfloor n/2 \rfloor, \ldots, \lfloor n/2 \rfloor + \alpha_n \sqrt{n} -1 \), we have

\[
\binom{n}{\ell} \cdot (n-\ell ) \geq \binom{n}{n/2+\alpha_n \sqrt{n}} \cdot (n - \lfloor n/2 \rfloor - \alpha_n \sqrt{n}) 
\]
and
\[
\sum_{\ell = \lfloor n/2 \rfloor }^{\lfloor n/2 \rfloor + \alpha_n\sqrt{n}-1}    \left|\mathcal{E}_{\ell}   - \mathcal{E}_{l+1} \right|
\geq \left|\mathcal{E}_{\lfloor n/2 \rfloor}   - \mathcal{E}_{\lfloor n/2 \rfloor + \alpha_n\sqrt{n}} \right| > \varepsilon
\]
we obtain
  \begin{equation*}
\begin{split}
\sum_{(w,w') \colon w \sim w'} (f_n(w) - f_n(w'))^2
&\geq 
\binom{n}{n/2+\alpha_n \sqrt{n}} \cdot (n - \lfloor n/2 \rfloor - \alpha_n \sqrt{n})   \cdot  \varepsilon.
\end{split}
\end{equation*}
Using that \(\left| \mathcal{E}_{\lfloor n/2 \rfloor + \alpha_n \sqrt{n}}- \mathcal{E}_{\lfloor n/2 \rfloor}   \right| > \varepsilon\), we obtain
\begin{equation*}
\begin{split}
\sum_{i=1}^n I_i(f_n)^2 
&\geq 
\sum_{i=1}^n \left( \frac{\sum_i I_i(f_n)}{n} \right)^2 
= 
n \cdot \left( \frac{ \sum_{w,w' \colon w \sim w'} (f(w) - f(w'))^2 }{\left| E(\Omega_{2,n}) \right| } \right)^2
\\&\geq
n \cdot \left( \frac{\binom{n}{n/2+\alpha_n \sqrt{n}} \cdot (n - \lfloor n/2 \rfloor - \alpha_n \sqrt{n})   \cdot  \varepsilon}{n2^n} \right)^2
\\&=
\varepsilon^2 
\cdot  \left( \frac{\sqrt{n}}{2^n} \cdot \binom{n}{n/2+\alpha_n \sqrt{n}}    \right)^2 
\cdot  \left( \frac{ n - \lfloor n/2 \rfloor - \alpha_n \sqrt{n}    }{n} \right)^2\!\!\!.
\end{split}
\end{equation*}

By applying Stirlings formula for \( n  \) sufficiently large, this contradicts that \( {\lim_{n \to \infty} \sum_{i=1}^n I_i(f_n)^2 = 0} \), from which the desired conclusion follows.
\end{proof}

We are now ready to give a proof of  Proposition~\ref{proposition: exclusion sensitivity}.
\begin{proof}[Proof of Proposition~\ref{proposition: exclusion sensitivity}]
Note first that by Lemma~\ref{lemma: splitting}, for any \( n \geq 1 \), \( f\colon \mathbb{Z}_2^n \to \{ 0,1 \} \),
\begin{equation*}
\begin{split}
&\Cov (f_n(X_0^{(n)}), f_n(X^{(n)}_{t_{rel}^{(n)}}))
\\&\hspace{2em}=
\sum_{m=0}^n P(\| X_0^{(n)} \| = m) \cdot \Biggl\{ \E \left[ f(X_0^{(n)})f(X^{(n)}_{t_{rel}^n}) \mid \| X^{(n)}_0 \| = m \right] 
\\& \hspace{18em}- \E \left[f(X_0^{(n)}) \mid \| X_0^{(n)} \| = m\right]^2 \Biggr\}
\\&\hspace{4em}+
\Var \left(  \E \left[ f(X_0^{(n)}) \mid   \| X_0^{(n)} \| \right] \right).
\end{split}
\end{equation*}
Given the assumption that \( \lim_{n \to \infty} \sum_{i=1}^n I_i(f_n)^2 = 0 \), Lemma~\ref{lemma: variance} implies that 
\[
 \lim_{n \to 0} \Var \left(  \E \left[ f(X_0^{(n)}) \mid   \| X_0^{(n)} \| \right] \right)= 0 .
\]
This in turn  implies that, to obtain the desired result,  it suffices to show that 
\begin{align*}
\lim_{n \to \infty}
\sum_{m=0}^n P(\| X_0^{(n)} \| = m) \cdot &\Biggl\{ \E \left[ f(X^{(n)}_0)f(X^{(n)}_{t_{rel}^{(n)}}) \mid \| X^{(n)}_0 \| = m \right] 
\\&\hspace{6em}- \E[f(X^{(n)}_0) \mid \| X^{(n)}_0 \| = m]^2 \Biggr\}
=0.
\end{align*}
By Lemma~\ref{lemma: good indices}, for any \( \alpha \in (0,0.5) \), if we let 
\[
 A_n(\alpha) = \{ m \in \{ 0,1, \ldots, n \} \colon \sum_{(ij)}  I_{(ij)}^{(m)}(f_n)^2<  n^{1 - \delta(1 - 2\alpha)}  \}
\]
then
\begin{equation*}
P \left(\| X_0^{(n)} \| \in A_n(\alpha) \right) \geq 1 - \frac{4}{n^{\alpha\delta}} .
\end{equation*}
Combining this with Lemma~\ref{lemma: common bound}, we obtain that for any \( \varepsilon > 0 \) and \( C > 0 \), there is \( C' = C'(C, \delta, \varepsilon) \) such that
\begin{equation*}
\begin{split}
\hspace{3em}&\hspace{-3em}\sum_{m=0}^n P(\| X_0^{(n)} \| = m) \cdot \Biggl\{ \E \left[ f(X_0^{(n)})f(X^{(n)}_{t_{rel}^{(n)}}) \mid \| X_0^{(n)} \| = m \right] 
\\[-2ex]&\hspace{14em}- \E[f(X_0^{(n)}) \mid \| X_0^{(n)} \| = m]^2 \Biggr\}
\\&\leq
\sum_{\mathclap{\substack{m \in A_n(\alpha) \colon \\ 2m(n-m) \geq \varepsilon n(n-1) }}} P(\| X_0^{(n)} \| = m) \cdot \Biggl\{ \E \left[ f(X_0^{(n)})f(X^{(n)}_{t_{rel}^n}) \mid \| X^{(n)}_0 \| = m \right] 
\\[-4ex]&\hspace{14em}
- \E[f(X^{(n)}_0) \mid \| X^{(n)}_0 \| = m]^2 \Biggr\}
\\&\hspace{2em}+ P(A_n^c)
+ P( 2\| X_0^{(n)} \| (n-\| X_0^{(n)} \|) < \varepsilon n(n-1))
\\&\leq
C' \cdot n^{1-\frac{1+2\delta(1-2\alpha)}{1 + \delta(1-2\alpha)}} + e^{- C}
+ P(A_n^c)
\\&\hspace{2em}+ P( 2\, \| X_0^{(n)} \| \cdot (n-\| X_0^{(n)} \|) \geq \varepsilon n(n-1))
\\&\leq
C' \cdot n^{1-\frac{1+2\delta(1-2\alpha)}{1 + \delta(1-2\alpha)}} + e^{- C}
+ 4n^{-\alpha \delta}
\\&\hspace{2em}+ P( 2\, \| X_0^{(n)} \| \cdot  (n-\| X_0^{(n)} \|) < \varepsilon n(n-1)).
\end{split}
\end{equation*}
As \( \| X_0^{(n)} \| \sim \bin(n,1/2) \), \( C \) can be chosen arbitrarily large and \( C' \) does not depend on \( n \), this can be made arbitrarily small by chosing \( n \) large. This finishes the proof.

\end{proof}

\section*{Acknowledgements}
The author would like to thank her supervisor Jeffrey Steif  for suggesting  the problem and for helpful discussions,  and Erik Broman for valuable comments on the manuscript.

\bibliographystyle{plain}
\bibliography{references}

\end{document}